\pdfoutput=1
\documentclass[a4paper,11pt,oneside]{amsart}
\usepackage[paperheight=279mm,paperwidth=18cm,textheight=26cm,textwidth=14cm]{geometry}
\usepackage{amssymb,amsmath,amsthm}
\usepackage{esint}
\usepackage{mathtools}
\usepackage[utf8]{inputenc}
\usepackage[T1]{fontenc}
\usepackage[unicode,hidelinks,bookmarks]{hyperref}
\hypersetup{
colorlinks=true,
citecolor=[rgb]{0,0,0.5},
linkcolor=[rgb]{0,0,0.5},
urlcolor=[rgb]{0,0,0.75},
pdfpagemode=UseNone,
pdfstartview=FitH,
pdfdisplaydoctitle=true,
pdftitle={Kakeya-Brascamp-Lieb},
pdfauthor={Pavel Zorin-Kranich},
pdflang=en-US
}
\usepackage[style=alphabetic]{biblatex}
\numberwithin{equation}{section}
\newtheorem{theorem}{Theorem}
\numberwithin{theorem}{section}
\newtheorem{definition}[theorem]{Definition}
\newtheorem*{theorem*}{Theorem}
\newtheorem{proposition}[theorem]{Proposition}
\newtheorem{lemma}[theorem]{Lemma}
\newtheorem{corollary}[theorem]{Corollary}

\theoremstyle{remark}
\newtheorem{remark}[theorem]{Remark}

\DeclarePairedDelimiter\abs{\lvert}{\rvert}
\DeclarePairedDelimiter\norm{\lVert}{\rVert}
\DeclarePairedDelimiter\floor{\lfloor}{\rfloor}

\providecommand\given{\colon}
\newcommand\SetSymbol[1][]{%
\nonscript\:#1\vert
\allowbreak
\nonscript\:
\mathopen{}}
\DeclarePairedDelimiterX\Set[1]\{\}{%
\renewcommand\given{\SetSymbol[\delimsize]}
#1
}

\newcommand*{\N}{\mathbb{N}}
\newcommand*{\R}{\mathbb{R}}
\def\<{\left\langle}
\def\>{\right\rangle}

\DeclareMathOperator{\supp}{supp}
\DeclareMathOperator{\dist}{dist}
\newcommand{\Vis}{\mathrm{Vis}}
\newcommand{\vol}{\operatorname{vol}}
\newcommand{\B}{\mathbb{B}}
\newcommand{\Sphere}{\mathbb{S}}
\newcommand{\s}{\mathfrak{s}}
\newcommand{\Deg}{D}
\newcommand*{\Poly}[1][\Deg]{\mathcal{P}_{#1}}

\newcommand{\Fotimes}{\bar{\otimes}}
\newcommand*{\widevec}[1]{\overrightarrow{#1}}
\newcommand{\one}{\mathbf{1}}
\newcommand{\dif}{\mathop{}\!\mathrm{d}} % \mathop produces two thin spaces, \! removes the trailing one
\newcommand{\BL}{\mathrm{BL}}
\newcommand{\calQ}{\mathcal{Q}}
\newcommand{\calT}{\mathcal{T}}
\newcommand{\codim}{\operatorname{codim}}

\begin{document}
\allowdisplaybreaks[2]
\title{Kakeya--Brascamp--Lieb inequalities}
\author{Pavel Zorin-Kranich}
\address{University of Bonn\\
  Mathematical Institute\\
  Bonn\\
  Germany
}
\begin{abstract}
We prove a sharp common generalization of endpoint multilinear Kakeya and local discrete Brascamp--Lieb inequalities.
\end{abstract}
\maketitle

\section{Introduction}

\subsection{Brascamp--Lieb inequalities}
\emph{Brascamp--Lieb (BL) inequalities} are a family of Lebesgue space estimates for multilinear forms that generalize H\"older's inequality, Young's convolutions inequality, and the Loomis--Whitney inequality.
Brascamp and Lieb proved that these inequalities are extremized by Gaussians \cite{MR0412366,MR1069246}, while Bennett, Carbery, Christ, and Tao gave a combinatorial characterization of the geometric situations in which these inequalities hold \cite{MR2377493,MR2661170}.
The latter characterization was recently extended to certain regularized Brascamp--Lieb inequalities, of a type relevant to Fourier decoupling theory, by Maldague \cite{arxiv:1904.06450}.

\begin{definition}
\label{def:BL}
We say that a function on a Euclidean space is \emph{constant at scale $r$} if the space can be partitioned in cubes with side length $r$, on each of which the function is constant.

Let $\vec{T} = (T_{1},\dotsc,T_{m})$ be a tuple of linear subspaces of $\R^{n}$, and let $\vec{p} = (p_{1},\dotsc,p_{m})$ be a tuple of positive real numbers.
For $0<r<R<\infty$ the \emph{Brascamp--Lieb constant} $\BL(\vec{T},\vec{p},(r,R))$ is the smallest constant $A$ such that the inequality
\begin{equation}
\label{eq:BL}
\int_{B(0,R)} \prod_{j=1}^{m} (f_{j}(x+T_{j}))^{p_{j}} \dif x
\leq
A \prod_{j=1}^{m} \bigl( \int_{\R^{n}/T_{j}} f_{j} \bigr)^{p_{j}}
\end{equation}
holds for every tuple of integrable functions $f_{j}$ on $\R^{n}/T_{j}$ that are constant at scale $r$.
\end{definition}

More general truncations and regularity conditions on the functions $f_{j}$ were studied in \cite[\textsection 8]{MR2377493}.
It is known from \cite[Lemma 8.14 and Corollary 8.15]{MR2377493} that the analogue of the inequality \eqref{eq:BL} with Gaussian cutoffs is extremized by Gaussian inputs.
The following result characterizes the growth rate of Brascamp--Lieb constants.

\begin{theorem}[{\cite{arxiv:1904.06450}}]
\label{thm:Maldague}
For every tuple of linear subspaces $\vec{T}$ and exponents $\vec{p}$ as in Definition~\ref{def:BL} there exist constants $0<c<C<\infty$ such that for every $0<r<R<\infty$ we have
\begin{equation}
\label{eq:Maldague}
c R^{\kappa}r^{\tilde{\kappa}} \leq \BL(\vec{T},\vec{p},(r,R)) \leq C R^{\kappa}r^{\tilde{\kappa}},
\end{equation}
where
\begin{align}
\label{eq:disBL-exponent}
\kappa &:=
\sup_{V \leq \R^{n}} \dim V - \sum_{j=1}^{m} p_{j} \dim (V/T_{j}),
\\
\label{eq:locBL-exponent}
\tilde{\kappa}
&:=
\bigl(n-\sum_{j=1}^{m}p_{j}n_{j}\bigr) - \kappa
=
\inf_{V \leq \R^{n}} \codim V - \sum_{j=1}^{m} p_{j} \codim (V/T_{j}).
\end{align}
Here $n_{j} := n - \dim T_{j}$, and the supremum and infimum are taken over all linear subspaces of $\R^{n}$.
\end{theorem}
Theorem~\ref{thm:Maldague} is stated in \cite{arxiv:1904.06450} for $r=1$.
This case implies the estimate \eqref{eq:Maldague} for general $r$, since by scaling we have
\begin{equation}
\label{eq:BL-scaling}
\BL( \vec{T}, \vec{p}, (r,R)) = r^{n-\sum_{j=1}^{m} p_{j} n_{j}} \BL( \vec{T}, \vec{p}, (1,R/r)).
\end{equation}
We briefly recall several special cases of Theorem~\ref{thm:Maldague} known prior to \cite{arxiv:1904.06450}.

\subsubsection{Discrete BL inequalities}
The exponent \eqref{eq:disBL-exponent} is always $\geq 0$, as can be seen by inserting $V=\Set{0}$.
The exponent \eqref{eq:disBL-exponent} vanishes if and only if for every subspace $V \leq \R^{n}$ we have
\begin{equation}
\label{eq:BCCT-discrete-cond}
\dim V \leq \sum_{j=1}^{m} p_{j} \dim (V/T_{j}).
\end{equation}
Hence the \emph{discrete BL constant} $\BL_{\mathrm{dis}}(\vec{T},\vec{p}) := \lim_{R\to \infty} \BL(\vec{T},\vec{p},(1,R))$ is finite if and only if \eqref{eq:BCCT-discrete-cond} holds.
This is the characterization of discrete BL inequalities originally proved in \cite[Theorem 2.5]{MR2661170}.

\subsubsection{Local BL inequalities}
\label{sec:local-BL}
The exponent \eqref{eq:locBL-exponent} is always $\leq 0$, as can be seen by inserting $V = \R^{n}$.
The exponent \eqref{eq:locBL-exponent} vanishes if and only if for every linear subspace $V \leq \R^{n}$ we have
\begin{equation}
\label{eq:BCCT-local-cond}
\codim V \geq \sum_{j=1}^{m} p_{j} \codim (V/T_{j}).
\end{equation}
Hence the \emph{local BL constant} $\BL_{\mathrm{loc}}(\vec{T},\vec{p}) := \lim_{r\to 0} \BL(\vec{T},\vec{p},(r,1))$ is finite if and only if \eqref{eq:BCCT-local-cond} holds.
This is the characterization of local BL inequalities originally proved in \cite[Theorem 8.17]{MR2377493} and \cite[Theorem 2.2]{MR2661170}.

\subsubsection{Scale-invariant BL inequalities}
\label{sec:scale-invariant-BL}
The \emph{scale-invariant BL constant}
\begin{equation}
\label{eq:scale-invariant-BL}
\BL(\vec{T},\vec{p}) := \lim_{r\to 0, R\to \infty} \BL(\vec{T},\vec{p},(r,R))
\end{equation}
is finite if and only if both exponents \eqref{eq:disBL-exponent} and \eqref{eq:locBL-exponent} vanish, which happens if and only if the conditions \eqref{eq:BCCT-discrete-cond} and \eqref{eq:BCCT-local-cond} hold.
The latter two conditions imply the scaling condition
\begin{equation}
\label{eq:BCCT-scaling-cond}
n = \sum_{j=1}^{m} p_{j} n_{j}.
\end{equation}
In fact, any two of the conditions \eqref{eq:BCCT-discrete-cond}, \eqref{eq:BCCT-local-cond}, and \eqref{eq:BCCT-scaling-cond} imply the third.
The fact that these conditions characterize the finiteness of the scale-invariant BL constant was proved in \cite[Theorem 1.15]{MR2377493}.

\subsubsection{Loomis--Whitney inequality}
\label{sec:LW}
One of the earliest and most important (scale-invariant) BL inequalities is the Loomis--Whitney inequality \cite{MR0031538}, which relies only on the product structure of $\R^{n}$ and not on the geometry of $\R$.
In its affine invariant form, it states that for $m=n$, $p_{1}=\dotsb=p_{m}=1/(m-1)$, and $k_{1}=\dotsb=k_{m}=1$, where $k_{j} := \dim T_{j}$, we have
\begin{equation}
\label{eq:LW}
\BL( \vec{T}, \vec{p} ) = \abs[\Big]{\bigwedge_{j=1}^{m} T_{j}}^{-1/(m-1)}.
\end{equation}
Here we identify each linear subspace $T_{j} \leq \R^{n}$ with the associated normalized volume form (the arbitrary choice of orientation does not affect the right hand-side of \eqref{eq:LW}).
It is easy to see that \eqref{eq:LW} continues to hold in the case $2 \leq m \leq n$, $p_{1}=\dotsb=p_{m}=1/(m-1)$, $k_{1}+\dotsb+k_{m}=n$; the corresponding multilinear forms are sometimes called \emph{$k$-plane transforms}.
More general Brascamp--Lieb inequalities that rely only on product structure were established by Finner \cite{MR1188047}.

\subsection{Multilinear Kakeya inequalities}
Kakeya--Brascamp--Lieb (KBL) inequalities are generalizations of Brascamp--Lieb inequalities \eqref{eq:BL} in which functions $f_{j}$, constant in directions of the subspaces $T_{j}$, are replaced by sums of functions constant in varying directions.
The first KBL inequality was the multilinear Kakeya inequality proved by Bennett, Carbery, and Tao \cite{MR2275834}.
Their work was motivated by applications to the Fourier restriction problem and generalized the Loomis--Whitney inequality (see Section~\ref{sec:LW}).
A simplified proof of a slightly weaker result, which is still adequate for all known applications, was given by Guth \cite{MR3300318}.

The multilinear Kakeya inequality of Bennett, Carbery, and Tao plays a central role in the proof of the $\ell^{2}$ decoupling theorem for the paraboloid by Bourgain and Demeter \cite{MR3374964} (see also \cite{MR3592159} for a streamlined exposition, in which the multilinear Kakeya inequality is applied directly, without passing through the multilinear restriction theorem).
More general KBL inequalities were first proved in \cite[Theorem 1.2]{MR3783217} and applied in the proof of the decoupling theorem for the moment curve by Bourgain, Demeter, and Guth \cite{MR3548534}.
We refer to \cite{arxiv:1811.02207} and references therein for the current state of development in this direction.

The proofs of the multilinear Kakeya inequality in \cite{MR2275834,MR3300318} proceed by induction on scales and do not yield the optimal endpoint result.
The endpoint version of the multilinear Kakeya inequality was established by Guth \cite{MR2746348} using a version of the polynomial method.
A more precise result that involves the transversality parameters \eqref{eq:LW} was obtained in \cite[Theorem 6]{MR2860188}.
Guth's proof was later simplified by Carbery and Valdimarsson \cite{MR3019726}.
Most recently, Ruixiang Zhang \cite{MR3738255} obtained endpoint KBL inequalities that extend all scale-invariant BL inequalities (see Section~\ref{sec:scale-invariant-BL}).

\subsection{Main result}
Our main result below tells that local discrete Brascamp--Lieb inequalities in Definition~\ref{def:BL} can be extended to Kakeya--Brascamp--Lieb inequalities with a loss in the constant that depends only on the dimension of the integration domain and the total power in the inequality.
\begin{theorem}
\label{thm:main:uniform}
Let $n\in\N = \Set{1,2,\dotsc}$, $m\geq 2$, $k_{1},\dotsc,k_{m} \in \N$, and $0<p_{1},\dotsc,p_{m}<\infty$ with
\begin{equation}
\label{eq:P}
P := \sum_{j=1}^{m} p_{j} \geq 1.
\end{equation}
For each $1\leq j\leq m$ let $\calT_{j}$ be a family of $k_{j}$-dimensional affine subspaces of $\R^{n}$.
Let $R>1$ and assume that for some $A < \infty$ and every tuple $\vec{T} = (T_{1},\dotsc,T_{m})$ with $T_{j} \in \calT_{j}$ we have
\begin{equation}
\label{eq:uniform-BL-bound}
\BL( \vec{T}, \vec{p}, (1,R) ) \leq A.
\end{equation}
Then for any functions $f_{j,T_{j}}$ constant at scale $1$ we have
\begin{equation}
\label{eq:main:uniform}
\int_{B(0,R)} \prod_{j=1}^{m} \Bigl( \sum_{T_{j} \in \calT_{j}} f_{j,T_{j}}(x + T_{j}) \Bigr)^{p_{j}} \dif x
\leq C^{P} A
\prod_{j=1}^{m} \Bigl( \sum_{T_{j} \in \calT_{j}} \int_{\R^{n}/T_{j}} f_{j,T_{j}} \Bigr)^{p_{j}}.
\end{equation}
\end{theorem}
Here and later $C$ denotes a constant that depends only on the dimension $n$, but may change between uses.
We write $A\lesssim B$ if $A\leq CB$.

The restriction $P\geq 1$ in Theorem~\ref{thm:main:uniform} is harmless, since BL inequalities are only nontrivial in this regime.

In the Loomis--Whitney case, Theorem~\ref{thm:main:uniform} reduces to \cite[Theorem 1.3]{MR2746348}.
In the scale-invariant case, Theorem~\ref{thm:main:uniform} refines \cite[Theorem 1.11]{MR3738255}, where a linear dependence of the estimate \eqref{eq:main:uniform} on the bound in \eqref{eq:uniform-BL-bound} was obtained for \emph{rational} exponents $p_{j}$ (this is not explicit in the statement of \cite[Theorem 1.11]{MR3738255}, but is clear from the statement of \cite[Theorem 8.1]{MR3738255} from which it is deduced).
Theorem~\ref{thm:main:uniform} is potentially useful in Fourier decoupling theory, see the discussion of the so-called ``ball inflation'' in \cite{arxiv:1811.02207}.

In the Loomis--Whitney case, the uniform bound \eqref{eq:uniform-BL-bound} can be obtained using the explicit description of the BL constants in \eqref{eq:LW}.
In the more general situation of Definition~\ref{def:BL}, it is known that, for a fixed tuple of exponents $\vec{p}$, the upper bound in \eqref{eq:Maldague} is locally uniform in $\vec{T}$.
This was proved in \cite[Theorem 2.1]{arxiv:1508.07502} for local BL constants and in \cite{arxiv:1904.06450} in the situation of Definition~\ref{def:BL}.

That local uniformity was made more precise in \cite{MR3723636}, where it is proved that, for fixed exponents $\vec{p}$, the scale-invariant BL constant $\BL(\vec{T}, \vec{p})$ depends continuously on the tuple of subspaces $\vec{T}$.
In \cite{MR3777414} this continuity result was further refined at rational points.

\subsection{A more precise KBL inequality}
Although Theorem~\ref{thm:main:uniform} is adequate for applications to decoupling, the uniform bound \eqref{eq:uniform-BL-bound} is not a natural hypothesis.
In the Loomis--Whitney setting more precise results, which remain meaningful in absence of such a uniform bound, were proved in \cite[Theorem 6]{MR2860188} and \cite[Theorem 1]{MR3019726}.
These results were extended to the $k$-plane transform case in \cite[Theorem 1.5]{MR3738255}.
In order to extend these results to general BL data, we will use the notion of Fremlin projective tensor product.

The Fremlin projective tensor product norm is the largest lattice norm on the algebraic tensor product of Banach lattices for which the inequality
\[
\norm{F_{1} \otimes \dotsm \otimes F_{m}} \leq \norm{F_{1}} \dotsm \norm{F_{m}}
\]
holds.
We will only need a special case of this notion, in which it has a simple explicit description.
\begin{definition}
\label{def:otimes}
Let $X_{1},\dotsc,X_{m}$ be measure spaces and $1<q_{1},\dotsc,q_{m}<\infty$ exponents with
\[
\sum_{j=1}^{m} \frac{1}{q_{j}} = 1.
\]
The \emph{Fremlin projective tensor product} $\Fotimes_{j=1}^{m} L^{q_{j}}(X_{j})$ is the space of measurable functions $F : X_{1} \times \dotsm \times X_{m} \to \R$ such that the norm
\begin{equation}
\label{eq:Fotimes}
\norm{F}_{\Fotimes_{j=1}^{m} L^{q_{j}}(X_{j})}
:=
\inf_{F_{j} \in L^{q_{j}}(X_{j}) \ :\  F_{j} \geq 0, \ \abs{F} \leq F_{1} \otimes\dotsm \otimes F_{m}}
\prod_{j=1}^{m} \norm{F_{j}}_{L^{q_{j}}(X_{j})}
\end{equation}
is finite.
\end{definition}
Subadditivity of the functional \eqref{eq:Fotimes} is a consequence of H\"older's inequality with exponents $q_{1},\dotsc,q_{m}$, see \cite[Theorem 2.2]{MR761991}.

We denote by $\calQ$ the grid of dyadic cubes in $\R^{n}$ with side length $1$, and for $R>1$ we denote by $\calQ_{R}$ the set of $Q\in\calQ$ at distance at most $R$ from the origin.

\begin{theorem}
\label{thm:main:aff-subspaces}
Let $n\in\N = \Set{1,2,\dotsc}$, $m\geq 2$, $k_{1},\dotsc,k_{m} \in \N$, and $0<p_{1},\dotsc,p_{m}<\infty$ with $P := \sum_{j=1}^{m} p_{j} \geq 1$.
For each $1\leq j\leq m$ let $\calT_{j}$ be a family of $k_{j}$-dimensional affine subspaces of $\R^{n}$, not necessarily distinct.
Then for every $R>1$ we have
\begin{equation}
\label{eq:main:affine-subspaces}
\sum_{Q\in\calQ_{R}} \norm{ \BL(\vec{T}, \vec{p}, (1,R))^{-1/P} }_{\Fotimes_{j=1}^{m} \ell^{P/p_{j}} (\Set{T_{j} \in \calT_{j} \given T_{j}\cap Q \neq \emptyset})}^{P}
\leq C^{P}
\prod_{j=1}^{m} \abs{\calT_{j}}^{p_{j}}.
\end{equation}
\end{theorem}
Here $\abs{\calT_{j}}$ denotes the cardinality of the family $\calT_{j}$, taking into account multiplicity.
The norm on the left-hand side of \eqref{eq:main:affine-subspaces} can be written more explicitly as
\begin{align*}
\MoveEqLeft
\norm{ \BL(\vec{T}, \vec{p}, (1,R))^{-1/P} }_{\Fotimes_{j=1}^{m} \ell^{P/p_{j}} (\Set{T_{j} \in \calT_{j} \given T_{j}\cap Q \neq \emptyset})}^{P}
\\ &=
\inf_{F_{j} \geq 0 \,:\, \BL(\vec{T}, \vec{p}, (1,R))^{-1/P} \leq F_{1}(T_{1}) \dotsm F_{m}(T_{m})}
\prod_{j=1}^{m} \norm{F_{j}}_{\ell^{P/p_{j}}(\Set{T_{j} \in \calT_{j} \given T_{j}\cap Q \neq \emptyset})}^{P}
\\ &=
\inf_{F_{j} \geq 0 \,:\, \BL(\vec{T}, \vec{p}, (1,R))^{-1} \leq F_{1}(T_{1}) \dotsm F_{m}(T_{m})}
\prod_{j=1}^{m} \norm{F_{j}}_{\ell^{1/p_{j}}(\Set{T_{j} \in \calT_{j} \given T_{j}\cap Q \neq \emptyset})}.
\end{align*}
\begin{proof}[Proof of Theorem~\ref{thm:main:uniform} assuming Theorem~\ref{thm:main:aff-subspaces}]
If \eqref{eq:uniform-BL-bound} holds, then by monotonicity of the functional~\eqref{eq:Fotimes} we have
\begin{align*}
\MoveEqLeft
\norm{ \BL( \vec{T}, \vec{p}, (1,R) )^{-1/P} }_{\Fotimes_{j=1}^{m} \ell^{P/p_{j}} (\Set{T_{j} \in \calT_{j} \given T_{j}\cap Q \neq \emptyset})}^{P}
\\ &\geq
A \norm{ \one }_{\Fotimes_{j=1}^{m} \ell^{P/p_{j}} (\Set{T_{j} \in \calT_{j} \given T_{j}\cap Q \neq \emptyset})}^{P}
\\ &=
A \prod_{j=1}^{M} \abs{\Set{T_{j} \in \calT_{j} \given T_{j}\cap Q \neq \emptyset}}^{p_{j}}.
\end{align*}
Inserting this inequality in \eqref{eq:main:affine-subspaces}, we obtain \eqref{eq:main:uniform} for functions $f_{j}$ that are the characteristic functions of $1$-neighborhoods of $T_{j}$ in $\R^{n}/T_{j}$.
Since we allow repetitions of $T_{j}$'s, this result immediately extends to functions $f_{j}$ that are constant at scale $1$ and take only integer values.
By homogeneity this extends to functions taking rational values, and by the monotone convergence theorem to arbitrary functions constant at scale $1$.
\end{proof}

\subsubsection{Loomis--Whitney case}
In the case $q_{1}=\dotsb=q_{m}=m$ it is easy to see that the norm \eqref{eq:Fotimes} is larger than the $L^{m}$ norm on the product space:
\begin{equation}
\label{eq:Fotimes-vs-Lm}
\norm{F}_{\Fotimes_{j=1}^{m} L^{m}(X_{j})}
\geq
\norm{F}_{L^{m}(X_{1} \times \dotsm \times X_{m})}.
\end{equation}
Suppose that $2 \leq m \leq n$, $k_{1}+\dotsb+k_{m}=n$, and $p_{1}=\dotsb=p_{m}=1/(m-1)$.
Then by \eqref{eq:LW} we have
\begin{equation}
\label{eq:LW-local}
\BL( \vec{T}, \vec{p}, (1,R) )
\leq
\BL( \vec{T}, \vec{p} )
=
\abs[\Big]{\bigwedge_{j=1}^{m} T_{j}}^{-1/(m-1)}.
\end{equation}
From \eqref{eq:LW-local} and \eqref{eq:Fotimes-vs-Lm} it follows that
\begin{align*}
\MoveEqLeft
\norm{ \BL( \vec{T}, \vec{p}, (1,R) )^{-1/P} }_{\Fotimes_{j=1}^{m} \ell^{P/p_{j}} (\Set{T_{j} \in \calT_{j} \given T_{j}\cap Q \neq \emptyset})}^{P}
\\ &\geq
\norm{ \abs[\Big]{\bigwedge_{j=1}^{m} T_{j}}^{1/(P(m-1))} }_{\Fotimes_{j=1}^{m} \ell^{P/p_{j}} (\Set{T_{j} \in \calT_{j} \given T_{j}\cap Q \neq \emptyset})}^{P}
\\ &\geq
\norm{ \abs[\Big]{\bigwedge_{j=1}^{m} T_{j}}^{1/m} }_{\ell^{m}( \prod_{j=1}^{m} \Set{T_{j} \in \calT_{j} \given T_{j}\cap Q \neq \emptyset})}^{P}
\\ &=
\Bigl( \sum_{T_{1} \in \calT_{1} \given T_{1}\cap Q \neq \emptyset} \dotsm \sum_{T_{m} \in \calT_{m} \given T_{m}\cap Q \neq \emptyset} \abs[\Big]{\bigwedge_{j=1}^{m} T_{j}} \Bigr)^{1/(m-1)}.
\end{align*}
Inserting this estimate in \eqref{eq:main:affine-subspaces} and observing that the resulting estimate does not depend on $R$, we recover the following result.
\begin{corollary}[{\cite[Theorem 1.5]{arxiv:1510.09132}}]
\label{cor:loomis-whitney}
In the setting of Theorem~\ref{thm:main:aff-subspaces} suppose that $\sum_{j=1}^{m} k_{j} = n$.
Then
\begin{equation}
\label{eq:main:loomis-whitney}
\sum_{Q\in\calQ} \abs[\Big]{ \sum_{T_{1} \in \calT_{1} : T_{1}\cap Q \neq \emptyset} \dotsi \sum_{T_{m} \in \calT_{m} : T_{m}\cap Q \neq \emptyset} \abs[\Big]{\bigwedge_{j=1}^{m} T_{j}} }^{1/(m-1)}
\leq C
\prod_{j=1}^{m} \abs{\calT_{j}}^{1/(m-1)}.
\end{equation}
\end{corollary}

Corollary~\ref{cor:loomis-whitney} in turn recovers the results of \cite{MR2746348,MR3019726}.
In Guth's case \cite{MR2746348} all affine subspaces are one-dimensional (lines).
Carbery and Valdimarsson \cite{MR3019726} state a result for $m$ families of lines with $2\leq m\leq n$.
In order to recover their result in the case $m<n$, one can replace the $m$-th family of lines by a family of affine subspaces of dimension $n-m+1$ as follows.
For each line find an orthonormal basis $e_{1},\dotsc,e_{n}$ of $\R^{n}$ such that $e_{1}$ points in the direction of that line.
Replace the line by the collection of affine subspaces containing the line and spanned by $e_{1}$ and a subset of $\Set{e_{2},\dotsc,e_{n}}$ of cardinality $n-m$.

\subsection{KBL inequality for varieties}
Following \cite{MR2860188} and \cite{MR3738255}, we deduce Theorem~\ref{thm:main:aff-subspaces} from a corresponding result for algebraic varieties.

\begin{theorem}
\label{thm:main}
Let $H_{1},\dotsc,H_{m}$ be algebraic varieties in $\R^{n}$ with $\dim H_{j} = k_{j}$ and $0 < p_{1},\dotsc,p_{m} < \infty$ with $P := \sum_{j=1}^{m} p_{j} \geq 1$.
Then for every $R>1$ we have
\begin{equation}
\label{eq:main}
\sum_{Q\in\calQ_{R}} \norm{ \BL( \widevec{T_{x_{j}}H_{j}}, \vec{p}, (1,R) )^{-1/P} }_{\Fotimes_{j=1}^{m} L^{P/p_{j}}_{x_{j}}(H_{j} \cap Q)}^{P}
\leq C^{P}
\prod_{j=1}^{m} (\deg H_{j})^{p_{j}},
\end{equation}
where each $H_{j}$ is equipped with the normalized $k_{j}$-dimensional surface measure on smooth points, and where $T_{x}H$ denotes the tangent space to $H$ at $x$.
\end{theorem}
In the scale-invariant case, Theorem~\ref{thm:main} with rational exponents $p_{j}$ essentially reduces to \cite[Theorem 8.1]{MR3738255}, up to the estimate \eqref{eq:Fotimes-vs-Lm}.
\begin{proof}[Proof of Theorem~\ref{thm:main:aff-subspaces} assuming Theorem~\ref{thm:main}]
By scaling, Theorem~\ref{thm:main} implies a similar statement with $H_{j} \cap Q$ replaced by $H_{j} \cap 3Q$.

Translating the $T_{j}$'s in each family $\calT_{j}$ slightly, we may assume that none of them coincide.
Then their union $H_{j} := \cup_{T_{j} \in \calT_{j}} T_{j}$ is a variety of degree $\lesssim \abs{\calT_{j}}$, and almost all points in each $T_{j}$ are smooth points of $H_{j}$.
It remains to observe that, for a $k$-dimensional affine subspace $T\subset\R^{n}$ and a unit length dyadic cube $Q$, non-emptiness of the intersection $T \cap Q \neq \emptyset$ implies $\vol_{k} (T \cap 3Q) \gtrsim 1$.
\end{proof}

The proof of Theorem~\ref{thm:main} consists of two parts.
The first part is Guth's construction of varieties with prescribed bounds on directional surface area in many cubes, Theorem~\ref{thm:vis:scaled}.
Its proof, which is presented in Section~\ref{sec:poly}, is not new, but we make a few simplifications compared to \cite{MR3019726}.
The second part consists in exploiting Theorem~\ref{thm:vis:scaled}, which we do in Section~\ref{sec:prf}.
The outline of this argument is close to \cite{MR3738255}, but using the Fremlin tensor product we avoid a few technical difficulties of that article.

\subsection{Notation}

\subsubsection{Exterior algebra}
Let $V,W$ be real vector spaces.
Recall that $\Lambda^{k}V$ denotes the $k$-th exterior power of $V$.
A bilinear form $\<\cdot,\cdot\> : V\times W \to \R$ induces a bilinear form $\Lambda^{k}V \times \Lambda^{k}W \to \R$ via
\begin{equation}
\label{eq:inner-prod-on-Lambdak}
\<v_{1} \wedge\dotsb\wedge v_{k}, w_{1} \wedge\dotsb\wedge w_{k} \> := \det (\<v_{i},w_{j}\>)_{i,j=1}^{k}.
\end{equation}
In particular, an inner product on $V$ induces an inner product on $\Lambda^{k}V$, which in turn induces a norm.

Throughout this article $n$ will remain fixed.
We write $\Lambda^{k} := \Lambda^{k}\R^{n}$ and let
\[
\abs{\Lambda^{k}}
:=
\Set{ v_{1} \wedge\dotsb\wedge v_{k} \given v_{1},\dotsc,v_{k} \in \R^{n} } / \Set{\pm 1}
\]
denote the set of simple $k$-fold wedge products of vectors in $\R^{n}$ modulo multiplication by $\pm 1$.
We work with $\abs{\Lambda^{k}}$ in place of $\Lambda^{k}$ in order to avoid dependence on choices of orientation of various spaces.

The wedge product $\wedge$ is well-defined as a map $\abs{\Lambda^{k_{1}}}\times\abs{\Lambda^{k_{2}}} \to \abs{\Lambda^{k_{1}+k_{2}}}$, the Hodge star $\star$ is well-defined as a map from $\abs{\Lambda^{k}}$ to $\abs{\Lambda^{n-k}}$, and the absolute value of the inner product (induced by \eqref{eq:inner-prod-on-Lambdak} from the Euclidean inner product) is well-defined as a map from $\abs{\Lambda^{k}} \times \abs{\Lambda^{k}}$ to $[0,\infty)$.
In particular, also the norm $\abs{v}$ is well-defined on $\abs{\Lambda^{k}}$.
More geometrically, $\abs{v_{1} \wedge \dotsb \wedge v_{k}}$ is the $k$-dimensional volume of the parallelepiped spanned by $v_{1},\dotsc,v_{k}$.

These operations can be extended to measures as follows.
Given measures $\mu_{1},\mu_{2}$ on $\abs{\Lambda^{k_{1}}}$ and $\abs{\Lambda^{k_{2}}}$, we denote by $\mu_{1}\wedge\mu_{2}$ the measure on $\abs{\Lambda^{k_{1}+k_{2}}}$ that is the pushforward of $\mu_{1}\times\mu_{2}$ under $\wedge$.
Similarly, given a measure $\mu$ on $\abs{\Lambda^{1}}$, we denote by $\mu^{\wedge k}$ the measure on $\abs{\Lambda^{k}}$ that is the pushforward of the $k$-fold product $\mu \times \dotsm \times \mu$ under the map $(v_{1},\dotsc,v_{k}) \mapsto v_{1} \wedge \dotsm \wedge v_{k}$.
Given a measure $\mu$ on $\abs{\Lambda^{k}}$, we denote by $\star\mu$ the measure on $\abs{\Lambda^{n-k}}$ that is the pushforward of $\mu$ under $\star$.
Given measures $\mu_{1},\mu_{2}$ on $\abs{\Lambda^{k}}$, we denote by $\<\mu_{1},\mu_{2}\>$ the measure on $[0,\infty)$ that is the pushforward of $\mu_{1}\times\mu_{2}$ under the map $(v_{1},v_{2}) \mapsto \abs{\<v_{1},v_{2}\>}$.
Finally, we denote the first moment of a measure $\mu$ on $\abs{\Lambda^{k}}$ by $\abs{\mu} := \int \abs{v} \dif\mu(v)$.

For a $k$-dimensional affine subspace $T \subset \R^{n}$, abusing the notation, we also denote by $T$ the unique element of $\abs{\Lambda^{k}}$ corresponding to the normalized Euclidean volume form on $T$.
Since the volume form is unique up to the choice of orientation, $T$ is well defined as an element of $\abs{\Lambda^{k}}$.

For a $k$-dimensional variety $H$ in $\R^{n}$, we denote by $T_{Q}H$ the push-forward of the surface measure on smooth points of $H\cap Q$ to $\abs{\Lambda^{k}}$ under the map that assigns to a point the normalized volume form of its tangent space.

\subsubsection{Directional area}
For a hypersurface $Z\subset\R^{n}$, let $NZ$ denote the pushforward of the normalized surface measure to $\abs{\Lambda^{1}}$ under the map that assigns to a smooth point $x\in Z$ the normalized normal vector $N_{x}Z$ at that point.
In particular, $NZ = \star TZ$.

Let $Z_{p} := \Set{x \given p(x)=0}$ denote the zero set of a polynomial.
For a probability measure $\sigma$ on the space of non-zero polynomials in $n$ variables we denote the average normal measure by
\begin{equation}
\label{eq:normal-measure}
\mu_{\sigma,Q} := \int N(Z_{p} \cap Q) \dif\sigma(p).
\end{equation}
This is a measure on the space $\abs{\Lambda^{1}}$.

To a measure $\mu$ on $\abs{\Lambda^{1}}$ we associate the seminorm
\begin{equation}
\label{eq:s_mu}
\s_{\mu}(v)
:=
\abs{\<v,\mu\>}
:=
\int_{\abs{\Lambda^{1}}} \abs{\<v,w\>} \dif\mu(w),
\quad
v \in \R^{n}.
\end{equation}
This is well-defined because the absolute value of the inner product is well-defined modulo $\Set{\pm 1}$.
In the case $\mu=NZ$ the seminorm \eqref{eq:s_mu} is the \emph{directed volume} \cite{MR2746348} or \emph{directional area} \cite{MR3019726} of $Z$ in the direction $v \in \R^{n}$.
We abbreviate
\[
\s_{\sigma,Q} := \s_{\mu_{\sigma,Q}}.
\]
We denote the Euclidean unit ball of $\R^n$ by $\B$.
For seminorm $\s$ on $\R^{n}$ we denote the associated unit ball by
\begin{equation}
\label{eq:Bs}
\B_{\s} := \Set{ v \in \R^{n} \given \s(v) \leq 1 }.
\end{equation}

\subsection*{Acknowledgment}
The extension of Theorem~\ref{thm:main:uniform} beyond the scale-invariant case \eqref{eq:BCCT-scaling-cond} is motivated by ongoing joint work with Shaoming Guo and Ruixiang Zhang.
This work was partially supported by the Hausdorff Center for Mathematics (DFG EXC 2047).
I thank the anonymous referees for numerous corrections and helpful suggestions pertaining to exposition.

\section{Polynomials with high visibility}
\label{sec:poly}
The following result is essentially due to Guth \cite{MR2746348}.
In this section we present a simplified proof, which is due to Carbery and Valdimarsson \cite{MR3019726}.
\begin{theorem}
\label{thm:vis:scaled}
For every $R>1$ there exists a positive integer
\begin{equation}
\label{deg}
\Deg \lesssim R
\end{equation}
such that for every function $M : \calQ_{R} \to [0,\infty)$ with $\sum_{Q \in \calQ_{R}} M(Q) = 1$ there exists a probability measure $\sigma$ on the space of non-zero polynomials in $n$ variables of degree at most $D$ such that for every $Q \in \calQ_{R}$ we have
\begin{equation}
\label{eq:Vis>M:scaled}
R^{n} M(Q) \vol \B_{\s_{\sigma,Q}} \lesssim 1
\end{equation}
and
\begin{equation}
\label{eq:Bmu-subset-B}
\B_{\s_{\sigma,Q}} \subset C\B.
\end{equation}
\end{theorem}

\subsection{Reduction to a global estimate}
\label{sec:vis}
Since we are mostly interested in the zero sets of polynomials, we will consider normalized polynomials.
Fix some norm on the space of all polynomials in $n$ variables with real coefficients, and let $\Poly$ denote the set of polynomials of degree at most $\Deg$ with norm $1$.
The set $\Poly$ is homeomorphic to the sphere of dimension $\binom{\Deg+n}{n} - 1 \sim \Deg^n$.

The measure $\sigma$ in Theorem~\ref{thm:vis:scaled} will be a mollification of a point measure.
Let
\begin{equation}
\label{eq:mu_p,Q}
\mu_{p,Q}
:=
\fint_{p'\in B(p,\epsilon)} N(Z_{p'}\cap Q),
\end{equation}
where $\epsilon>0$ will be a very small number to be chosen later (depending on the function $M$ in Theorem~\ref{thm:vis}), $B(p,\epsilon)$ is an $\epsilon$-ball on the sphere $\Poly$, and $\fint$ denotes the average with respect to the surface measure on $\Poly$.
We will omit the dependence of $\mu_{p,Q}$ on $\epsilon$ from the notation.
The mollification is necessary in order to ensure continuity of the map $p \mapsto \mu_{p,Q}$, which is needed in the Borsuk--Ulam theorem.
Specifically, continuity will be used in \eqref{eq:E-eventually-constant}.

We define the \emph{visibility} of a seminorm $\s$ on $\R^{n}$ by
\begin{equation}\label{Kdefn}
\Vis(\s) := (\vol K_{\s} )^{-1},
\quad
K_{\s} := \B \cap \B_{\s}.
\end{equation}
Since $K_{\s} \subseteq \B$, we always have $\Vis(\s) \gtrsim 1$.
The conventions for both visibility and $M$ in \cite{MR3019726} differ by a power $n$ from those in \cite{MR2746348}.
We use the conventions in \cite{MR2746348}.

The next result is the core of Guth's argument from \cite{MR2746348}.
Here and later we abbreviate
\[
\s_{p,Q} := \s_{\mu_{p,Q}},
\quad
K_{p,Q} := K_{\s_{p,Q}},
\quad
\Vis_{p,Q} := \Vis_{\s_{p,Q}}.
\]

\begin{theorem}[{\cite[Lemma 6.6]{MR2746348}}]\label{thm:vis}
There exists a dimensional constant $c=c(n)>0$ such that for every positive integer $\Deg$ and every function $M : \calQ \to [0,\infty)$ with
\begin{equation}
\label{eq:vis:degree}
\Bigl(\sum_{Q\in\calQ} M(Q)\Bigr)^{1/n} \leq c \Deg
\end{equation}
there exist an $\epsilon>0$ and a polynomial $p \in \Poly$ such that for every $Q \in \calQ$ we have
\begin{equation}
\label{eq:Vis>M}
\Vis_{p,Q} \geq M(Q).
\end{equation}
\end{theorem}

\begin{proof}[Proof of Theorem~\ref{thm:vis:scaled} assuming Theorem~\ref{thm:vis}]
Let
\[
p_{0}(x) := \prod_{c\in \mathbb{Z}+1/2 : \abs{c} \leq R+1} \prod_{i=1}^{n} (x_{i}-c).
\]
This is a polynomial of degree $\lesssim R$, and we have
\begin{equation}
\label{eq:p0-normal-measure}
\abs{ \< v, N(Z_{p_{0}}\cap Q) \> } \gtrsim \abs{v}
\end{equation}
for all $Q\in\calQ_{R}$ and all $v\in\R^{n}$.

Applying Theorem~\ref{thm:vis} with $D$ being a sufficiently large multiple of $R$ and
\[
\tilde{M}(Q) :=
\begin{cases}
R^{n}M(Q), & Q \in \calQ_{R},\\
0 & \text{otherwise,}
\end{cases}
\]
we obtain some $p\in\Poly$ and $\epsilon>0$ such that \eqref{eq:Vis>M} holds with $\tilde{M}$ in place of $M$.
Let $\sigma$ be the pushforward of the normalized measure on $B(p,\epsilon)$ under multiplication by $p_{0}$.
Since $\deg p_{0} \lesssim R$, the measure $\sigma$ is supported on a set of polynomials of degree $\lesssim R$.
Moreover, for every non-zero polynomial $p'$ we have
\[
N(Z_{p'p_{0}} \cap Q) = N(Z_{p'} \cap Q) + N(Z_{p_{0}} \cap Q).
\]
It follows from \eqref{eq:p0-normal-measure} that
\[
\s_{\sigma,Q}(v) \gtrsim \abs{v}
\]
for every $Q \in \calQ_{R}$ and every $v \in \R^{n}$, and this implies \eqref{eq:Bmu-subset-B}.

Furthermore, for $Q \in \calQ_{R}$ with $M(Q) \neq 0$ we have
\[
\B_{\s_{\sigma,Q}} \subseteq \B_{p,Q} \cap C\B,
\]
hence by \eqref{eq:Vis>M} we obtain
\[
\vol(\B_{\s_{\sigma,Q}})
\lesssim
\vol(C^{-1}\B_{p,Q} \cap \B)
\leq
\vol(K_{p,Q})
\leq
\tilde{M}(Q)^{-1},
\]
and this implies \eqref{eq:Vis>M:scaled}.
\end{proof}

\begin{remark}
The idea of replacing $p$ by $pp_{0}$ was introduced just after the formula (7.2) of \cite{MR2746348}.
A literal interpretation of \cite{MR2746348} suggests using $\sigma$ that is the normalized measure on $B(pp_{0},\epsilon)$.
In \cite[\textsection 4]{MR3019726} it was observed that this would require an additional continuity argument, and an alternative approach was given.
That alternative approach requires an application of Theorem~\ref{thm:vis} with $M$ replaced by $\lambda M$ with
\[
\lambda
\sim
\max_{Q\in \supp(M)} M(Q)^{-1}
\gtrsim
R^{n}.
\]
Consequently, it produces polynomials of a higher, and in general unbounded, degree.
The degree does not matter if the scaling condition \eqref{eq:BCCT-scaling-cond} holds, but it seems to become important in our more general situation.
The more convenient interpretation of the replacement $p \mapsto pp_{0}$ above, in which $\sigma$ is the normalized measure on $p_{0}B(p,\epsilon)$, was first made explicit in \cite[formula (6.6)]{MR3738255}, which in our notation corresponds to $\s_{\sigma,Q}$.
\end{remark}

Since for every polynomial $p\in\Poly$ and every cube $Q$ we have $\Vis_{p,Q} \gtrsim 1$, the condition \eqref{eq:Vis>M} has content only for those $Q$ with $M(Q) \gtrsim 1$.
Hence we may assume that $M$ is finitely supported and $M(Q)$ is either $0$ or $\gtrsim 1$ for every $Q$.

\subsection{The Borsuk--Ulam theorem and a covering lemma}\label{sec:BU}

The existence of a polynomial $p$ claimed in Theorem~\ref{thm:vis} can be equivalently stated by saying that the sets
\[
S(Q) := \Set{ p \in \Poly \given \Vis_{p,Q} \leq M(Q) }
\]
do not cover $\Poly$, provided that $\epsilon$ is sufficiently small.
In order to show this, we will write these sets as unions of sets to which Lemma~\ref{covering} below applies.

\begin{theorem}[Borsuk--Ulam]
\label{thm:borsuk+ulam}
Let $N \geq J$ and suppose that $F: \Sphere^N \to \R^J$ is a continuous function that satisfies $F(-x) = - F(x)$ for all $x \in \Sphere^N$.
Then there exists an $x \in \Sphere^N$ such that $F(x) = 0$.
\end{theorem}

See \cite{MR1988723} for a discussion of applications and various equivalent forms of this theorem.
Some of these equivalent forms (known as Lusternik--Schnirelmann results) take the form of covering statements for the sphere.
We will use one such statement.

\begin{lemma}[{\cite[\textsection 5]{MR3019726}}]\label{covering}
Suppose that $A_i \subseteq \Sphere^N$ for $1 \leq i \leq J$, and suppose that for each $i$,
$A_i \cap (\overline{-A_i}) = \emptyset$.
If $J \leq N$, then the $2J$ sets $A_i$ and $-A_i$ do not cover $\Sphere^N$.
\end{lemma}

Note that no topological hypothesis on the sets $A_i$ is needed.
 
\begin{proof}
Without loss of generality we may assume that all sets $A_{i}$ are non-empty.
Define $F: \Sphere^N \to \R^J$ by $F_i(x) := d(x, -A_i) - d(x, A_i)$ for $1 \leq i \leq J$ and $x\in\Sphere^{N}$.
Then $F$ is continuous and $F(-x) = -F(x)$ for all $x$, so by Theorem~\ref{thm:borsuk+ulam} there is an $x \in \Sphere^{N}$ with $F(x) = F(-x) = 0$.
We claim that this $x$ does not belong to any $A_i$ or to $-A_i$.
Indeed, if $x \in A_i$, then $d(x, A_i) = 0$, hence also $d(x, -A_i) = F_{i}(x) + d(x,A_{i}) = 0$, hence $x \in \overline{-A_i}$, a contradiction.
If $x \in -A_{i}$, then we obtain a contradiction by the same argument with $x$ replaced by $-x$.
\end{proof}

\subsection{Discretization of the set of norms}
In this section we discretize the set of all norms on $\R^{n}$.
The argument follows \cite[p.~278]{MR2746348}.

Let $\mathcal{K}$ denote the set of centrally symmetric convex bodies in $\R^n$ with the metric
\[
d(K,L) := \log \inf \Set{ \alpha \geq 1 \given \alpha^{-1}K \subseteq L \subseteq \alpha K}.
\]
This metric is $GL(n)$-invariant.
Let $\mathcal{E} \subset \mathcal{K}$ denote the class of centered ellipsoids in $\R^n$, that is, images of the unit ball $\B$ under invertible linear maps $A\in GL(n)$.
Notice that $d(A(\B),\B) = \log\max (\norm{A},\norm{A^{-1}})$ for all $A\in GL(n)$.
Hence $\mathcal{E} \cong GL(n)/O(n)$ both in the topological group and the metric sense, where $d(A,B) = \log\max (\norm{A^{-1}B}, \norm{B^{-1}A})$ for $A,B\in GL(n)$.
It follows that the metric space $(\mathcal{E},d)$ is homogeneous and locally compact.

By the John ellipsoid theorem \cite{MR0030135}, the set $\mathcal{E}$ is $C_{n}/2$-dense in $\mathcal{K}$ for $C_{n}=(\log n)/2$.

It follows from local compactness that the cardinality of a $C_{n}/2$-separated subset of the ball $\mathcal{B} := \Set{ A\in GL(n) \given d(A,\operatorname{Id}) \leq 4C_{n}}$ is bounded by some $D_{n}<\infty$.
By homogeneity, the cardinality of a $C_{n}/2$-separated subset of \emph{every} ball of radius $4C_{n}$ in $\mathcal{E}$ is bounded by $D_{n}$.
Let $\mathcal{E}_{0} \subset \mathcal{E}$ be a maximal $C_{n}/2$-separated subset of ellipsoids with diameter at most $1$.
For each $\theta=1,2,\dotsc$ let inductively $\mathcal{E}_{\theta}$ be a maximal $4C_{n}$-separated subset of $\mathcal{E}_{0} \setminus \cup_{1\leq \theta'<\theta} \mathcal{E}_{\theta'}$.
Then
\begin{equation}
\label{eq:colours}
\mathcal{E}_{0} = \bigcup_{\theta\in\Theta} \mathcal{E}_{\theta},
\quad
\Theta := \Set{1,\dotsc,D_{n}}.
\end{equation}
Indeed, if $E\in\mathcal{E}_{0} \setminus \cup_{1\leq \theta' \leq \theta} \mathcal{E}_{\theta'}$, then $\mathcal{E}_{\theta}\cap B(E,4C_{n}) \neq \emptyset$, since otherwise we could add $E$ to $\mathcal{E}_{\theta}$, contradicting maximality.
This can only occur for at most $D_{n}-1$ values of $\theta$ for each $E\in\mathcal{E}_{0}$.
The partition \eqref{eq:colours} has the following properties.

\begin{lemma}\label{ellipsoid}
For every $K\in\mathcal{K}$ with $K \subset \B$, there exists $E\in\mathcal{E}_{0}$ such that $d(K,E)\leq C_{n}$.
On the other hand, for every $\theta\in\Theta$ and every $K\in\mathcal{K}$, there is at most one $E\in\mathcal{E}_{\theta}$ such that $d(K,E) < 2C_{n}$.
\end{lemma}

Finally, given $Q \in \supp M$, let
\[
S(Q,\theta) := \Set{ p \in S(Q) \given d(K_{p,Q}, \mathcal{E}_{\theta}) \leq C_{n}},
\]
so that
\begin{equation}\label{colours}
S(Q) = \bigcup_{\theta \in \Theta} S(Q,\theta).
\end{equation}

\subsection{Bisecting balls}
The following fact is a consequence of the isoperimetric inequality.
\begin{lemma}[{\cite[Appendix A]{MR3019726}}]
\label{lem:bisect}
Let $p : \R^{n} \to \R$ be a polynomial,
\[
a := \frac{\abs{\Set{x \in \B \given p(x)>0}}}{\abs{\B}},
\quad
b := \frac{\abs{\Set{x \in \B \given p(x)<0}}}{\abs{\B}}.
\]
Then
\[
\mathcal{H}_{n-1}(Z_{p} \cap \B) \geq \frac12 (a^{(n-1)/n} + b^{(n-1)/n} - 1) \mathcal{H}_{n-1}(\mathbb{S}^{n-1}).
\]
\end{lemma}
An affine invariant formulation of Lemma~\ref{lem:bisect} is the following.
\begin{corollary}[{\cite[p.~1659]{MR3019726}}]
\label{cor:bisect}
Let $E\subset \R^{n}$ be an ellipsoid with principal axes $v_{1},\dotsc,v_{n}$ and $p : \R^{n} \to \R$ a polynomial that approximately bisects $E$ in the sense that
\[
\abs{\Set{x \in E \given p(x)>0}} > c\abs{E},
\quad
\abs{\Set{x \in E \given p(x)<0}} > c\abs{E}
\]
for some $c>0$.
Then
\begin{equation}
\label{eq:bisect-area-affinv}
\sum_{j=1}^{n} \abs{\<v_{j},N(Z_{p} \cap Q)\>} \gtrsim \vol(E).
\end{equation}
\end{corollary}
Recall that the left-hand side of \eqref{eq:bisect-area-affinv} was defined in \eqref{eq:s_mu}.

\subsection{Translates}\label{sec:translates}
For $p\in S(Q,\theta)$, let $E(p,Q,\theta)$ be the unique ellipsoid in $\mathcal{E}_{\theta}$ such that $d(K_{p,Q},E(p,Q,\theta)) \leq C_{n}$.
Let $0<\eta\ll 1$ be a parameter depending only on $n$ to be chosen later.

For each $E \in \mathcal{E}_0$ and $Q\in\calQ$, we can fit $\floor{c_{n} \eta^{-n} \vol(E)^{-1}}$ disjoint translates of $\eta E$ inside $Q$ (here we use that $E$ has diameter $\leq 1$).
We make a choice of such disjoint translates for every $E$ and denote them by $E_{\alpha,Q}$ with $1 \leq \alpha \leq \floor{c_{n} \eta^{-n} \vol(E)^{-1}}$.

A polynomial $p \not\equiv 0$ is said to \emph{bisect} a set $E'$ if $\abs{E' \cap \Set{p>0}} = \abs{E' \cap \Set{p<0}}$.
Let
\[
S(Q,\theta,\alpha) 
:= \begin{cases}
\Set{ p\in S(Q,\theta) \given p \text{ does not bisect } E(p,Q,\theta)_{\alpha,Q} }
&\text{if } \alpha \leq c_{n} \eta^{-n} \vol(E)^{-1},\\
\emptyset & \text{otherwise.}
\end{cases}
\]
\begin{lemma}[{\cite[\textsection 8]{MR3019726}}]
\label{manybisections}
There exist $\eta=\eta(n)>0$ and $\epsilon=\epsilon(n,M)>0$ such that for all $Q \in \calQ$ with $M(Q) \gtrsim 1$ and all $\theta\in\Theta$ we have
\[
S(Q,\theta) = \bigcup_{\alpha \lesssim M(Q)} S(Q,\theta,\alpha).
\]
\end{lemma}

Lemma~\ref{manybisections} states that each $p\in S(Q,\theta)$ does not bisect some $E_{\alpha,Q}$ with $E=E(p,Q,\theta)$.
In fact, one can obtain the stronger conclusion that $p$ does not even approximately bisect some $E_{\alpha,Q}$.

\begin{proof}
Let $\eta=\eta(n)$ to be chosen later.
Let $\epsilon = \epsilon(n,M)$ be so small that for all $Q\in\supp M$ and all $p,p'\in\Poly$ we have
\[
\dist(p,p') < \epsilon
\implies
\abs{(\Set{p>0} \Delta \Set{p'>0})\cap Q}
\leq
c' \eta^{n} M(Q)^{-1},
\]
where $c'$ is a small dimensional constant.
Such $\epsilon$ exists because the function $(p,p')\mapsto \abs{(\Set{p>0} \Delta \Set{p'>0})\cap Q}$ is continuous by the dominated convergence theorem, and therefore uniformly continuous since $\Poly$ is compact.

Let $Q \in \supp M$, $p\in S(Q,\theta)$, and $E:=E(p,Q,\theta)$.
Suppose for a contradiction that $p$ bisects $E_{\alpha,Q}$ for each $\alpha \leq c_{n}\eta^{-n}\vol(E)^{-1} \lesssim \eta^{-n} M(Q)$.

Let $v_{1},\dotsc,v_{n} \in \R^{n}$ be the principal axes of $E$ (\emph{not} normalized to unit length), so that each $E_{\alpha,Q}$ has principal axes $\eta v_{1},\dotsc,\eta v_{n}$.
If $p$ bisects $E_{\alpha,Q}$ and $p'\in\Poly$ is another polynomial $\epsilon$-close to $p$, then $p'$ still approximately bisects $E_{\alpha,Q}$ by the choice of $\epsilon$.
By Corollary~\ref{cor:bisect}, for each such $p'$ we have
\[
\sum_{j=1}^{n} \abs{ \<\eta v_j, N(Z_{p'} \cap E_{\alpha,Q})\>}
\gtrsim
\vol(E_{\alpha,Q})
\sim
\eta^{n} \vol E.
\]
Averaging in $p'$ and summing in $\alpha$, we obtain
\[
\sum_{j=1}^n \s_{p,Q}(\eta v_{j})
\gtrsim
\floor{c_{n}\eta^{-n}\vol(E)^{-1}} \eta^{n} \vol E
\]
with the implicit constant independent of $\eta$.
Since $d(E,K_{p,Q}) \leq C_{n}$, we have $\s_{p,Q}(v_{j}) \sim 1$, and it follows that $\eta \gtrsim 1$.
This leads to a contradiction if $\eta$ is small enough.
\end{proof}

\subsection{Antipodes}\label{Antipodes}
Fix $Q \in \supp M$, $\theta \in \Theta$, and $\alpha \geq 1$.
Since any $p\in S(Q,\theta,\alpha)$ does not bisect $E' := E(p,Q,\theta)_{\alpha,Q}$, we have either
\[
\vol (\Set{p>0} \cap E') > \vol (\Set{p<0 } \cap E'),
\]
in which case we say that $p \in S(Q,\theta,\alpha,+)$, or 
\[
\vol (\Set{p>0} \cap E') < \vol (\Set{p<0 } \cap E'),
\]
in which case we say that $p \in S(Q,\theta,\alpha,-)$.
In particular
\[
S(Q,\theta,\alpha) = S(Q,\theta,\alpha,+) \cup S(Q,\theta,\alpha,-)
\]
and
\begin{equation}\label{opposite}
S(Q,\theta,\alpha,+) = - S(Q,\theta,\alpha,-).
\end{equation}
It remains to show that
\begin{equation}\label{qw}
S(Q,\theta,\alpha,+) \cap \overline{S(Q,\theta,\alpha,-)} = \emptyset,
\end{equation}
where the closure is taken in the natural topology of $\Poly$.

To this end, suppose for a contradiction that $p$ lies in the intersection \eqref{qw}.
Then $p \in S(Q,\theta,\alpha,+)$, and there exists a sequence of $p_m \in S(Q,\theta,\alpha,-)$ that converges to $p$ in $\Poly$.
Since the function $(p,v)\mapsto \s_{p,Q}(v)$ is jointly continuous from $\Poly \times \R^{n}$ to $[0,\infty)$ (thanks to the mollification by $\epsilon$ in \eqref{eq:mu_p,Q} and the uniform bound on $\abs{N(Z_{p'}\cap Q)}$ given by Corollary~\ref{cor:tube}), the function $p\mapsto K_{p,Q}$ is continuous from $\Poly$ to $\mathcal{K}$.
It follows that
\begin{equation}
\label{eq:E-eventually-constant}
E(p_{m},Q,\theta)=E(p,Q,\theta)=:E
\end{equation}
for sufficiently large $m$.
Indeed, since $p,p_m\in S(Q,\theta)$, the sets $K_{p_m,Q}$ and $K_{p,Q}$ must be close to \emph{some} member of $\mathcal{E}_\theta$ and thus, for $m$ sufficiently large, they are close to the \emph{same} member of $\mathcal{E}_{\theta}$.
This is why we need several collections $\mathcal{E}_{\theta}$.

With $E$ given by \eqref{eq:E-eventually-constant}, we have
\begin{equation}\label{last}
\vol(\Set{p>0} \cap E_{\alpha,Q}) > \vol(\Set{p<0 } \cap E_{\alpha,Q}),
\end{equation}
and for sufficiently large $m$ we obtain
\begin{equation}\label{last:m}
\vol(\Set{p_m> 0} \cap E_{\alpha,Q}) < \vol(\Set{p_m <0 } \cap E_{\alpha,Q}).
\end{equation}
By the dominated convergence theorem this leads to a contradiction.

\subsection{Conclusion of the proof of Theorem~\ref{thm:vis}}
We have decomposed
\begin{equation}\label{union}
\bigcup_{Q\in \supp M} S(Q)
=
\bigcup_{Q\in \supp M} \bigcup_{\theta\in\Theta} \bigcup_{\alpha \lesssim M(Q)}
\left(S(Q,\theta,\alpha,+) \cup S(Q,\theta,\alpha,-) \right),
\end{equation}
where $S(Q,\theta,\alpha,+)$ and $S(Q,\theta,\alpha,-)$ are antipodal by \eqref{opposite} and separated by \eqref{qw}.
The union consists of $\lesssim \sum_{Q} M(Q)$ terms, so by Lemma~\ref{covering} it does not cover $\Poly$ provided $\Deg^{n} \sim \dim\Poly \gtrsim \sum_{Q} M(Q)$, which is the case by the hypothesis \eqref{eq:vis:degree}.

\section{Proof of Theorem~\ref{thm:main}}
\label{sec:prf}

\subsection{Multilinear duality}
Let us abbreviate
\begin{equation}
\label{eq:G}
G(Q) := \norm{ \BL(\widevec{T_{x_{j}}H_{j}}, \vec p, (1,R))^{-1/P} }_{\Fotimes_{j=1}^{m} L^{P/p_{j}}_{x_{j}}(H_{j} \cap Q)}^{P}.
\end{equation}
The following formulation of \eqref{eq:main} goes back to \cite[Proposition 2]{MR3019726}.
It motivated the development of multilinear factorization theory in \cite{arxiv:1809.02449}.
\begin{proposition}
\label{prop:multilinear-duality}
For any $0<C_{1},C_{2}<\infty$, the estimate
\begin{equation}
\label{eq:main:G}
\sum_{Q\in\calQ_{R}} G(Q)
\leq C_{1}^{P}C_{2}^{P}
\prod_{j=1}^{m} (\deg H_{j})^{p_{j}}
\end{equation}
is equivalent to the following statement:

For every function $M : \calQ_{R} \to [0,\infty)$ satisfying
\begin{equation}
\label{eq:tensor:norm}
\sum_{Q\in\calQ_{R}} M(Q) = 1,
\end{equation}
there exist functions $S_{j} : \calQ_{R} \to [0,\infty)$ such that
\begin{equation}
\label{eq:tensor:prod}
G(Q) M(Q)^{P-1}
\leq C_{1}^{P}
\prod_{j=1}^{m} S_{j}(Q)^{p_{j}}
\quad\text{for every } Q\in\calQ_{R}
\end{equation}
and
\begin{equation}
\label{eq:tensor:sum}
\sum_{Q \in \calQ_{R}} S_{j}(Q)
\leq C_{2}
\deg H_{j}.
\end{equation}
\end{proposition}
The advantage of this formulation is that a suitable choice of $S_{j}$, see \eqref{eq:Sj}, allows one to separation the rather geometric arguments used to establish \eqref{eq:tensor:prod} in Section~\ref{sec:intersection-multiplicity} from the rather algebraic arguments used to establish \eqref{eq:tensor:sum} in Section~\ref{sec:wedge}.

\begin{proof}[Proof of Proposition~\ref{prop:multilinear-duality}]
Suppose that there exist $S_{j}$ satisfying \eqref{eq:tensor:prod} and \eqref{eq:tensor:sum}; we want to show \eqref{eq:main:G}.
For every $Q\in \calQ_{R}$, applying \eqref{eq:tensor:prod} with $M(Q')=\one_{Q'=Q}$, we see that $G(Q)$ is finite.

Without loss of generality we may assume $\mathcal{G} := \sum_{Q \in \calQ_{R}} G(Q) > 0$, so that we can define $M(Q) := \one_{Q\in\calQ_{R}} G(Q)/\mathcal{G}$.
Then, by \eqref{eq:tensor:prod}, H\"older's inequality, and \eqref{eq:tensor:sum}, we obtain
\begin{align*}
\mathcal{G}
&=
\Bigl(\mathcal{G}^{-(P-1)/P} \sum_{Q \in \calQ_{R}} G(Q)\Bigr)^{P}
\\ &=
\Bigl( \sum_{Q \in \calQ_{R}} G(Q)^{1/P} M(Q)^{(P-1)/P} \Bigr)^{P}
\\ &\leq C_{1}^{P}
\Bigl( \sum_{Q \in \calQ_{R}} \prod_{j=1}^{m} S_{j}(Q)^{p_{j}/P} \Bigr)^{P}
\\ &\leq C_{1}^{P}
\prod_{j=1}^{m} \bigl( \sum_{Q \in \calQ_{R}} S_{j}(Q) \bigr)^{p_{j}}
\\ &\leq C_{1}^{P} C_{2}^{P}
\prod_{j=1}^{m} \bigl( \deg H_{j} \bigr)^{p_{j}}.
\end{align*}
The converse (that \eqref{eq:main:G} implies the existence of appropriate $S_{j}$'s) can be established with the ansatz $S_{j}(Q)=C_{2} S(Q) \deg H_{j}$ and
\[
S(Q) := \frac{G(Q)^{1/P} M(Q)^{1-1/P}}{\sum_{Q' \in \calQ_{R}} G(Q')^{1/P} M(Q')^{1-1/P}}.
\]
Then \eqref{eq:tensor:sum} is clearly satisfied, and \eqref{eq:tensor:prod} holds because by H\"older's inequality and \eqref{eq:main:G} we have
\begin{align*}
\Bigl( \sum_{Q \in \calQ_{R}} G(Q)^{1/P} M(Q)^{1-1/P} \Bigr)^{P}
& \leq
\Bigl( \sum_{Q \in \calQ_{R}} G(Q) \Bigr)
\Bigl( \sum_{Q \in \calQ_{R}} M(Q) \Bigr)^{P-1}
\\ &\leq
C_{1}^{P} C_{2}^{P} \prod_{j=1} (\deg H_{j})^{p_{j}}.
\qedhere
\end{align*}
\end{proof}

We will verify the conditions in Proposition~\ref{prop:multilinear-duality} in the setting of Theorem~\ref{thm:main}.
To this end let $M : \calQ_{R} \to [0,\infty)$ be an arbitrary function with $\sum_{Q\in\calQ_{R}} M(Q)=1$.
Theorem~\ref{thm:vis:scaled}, applied to this function $M$, gives a probability measure $\sigma$ supported on polynomials of degree $\lesssim R$ and satisfying \eqref{eq:Vis>M:scaled} and \eqref{eq:Bmu-subset-B}.

Let $\mu_{Q} := \mu_{\sigma,Q}$ (recall the definition \eqref{eq:normal-measure}) and define
\begin{equation}
\label{eq:Sj}
S_{j}(Q) := R^{-k_{j}} \abs[\Big]{\<T_{Q}H_{j}, \mu_{Q}^{\wedge k_{j}}\>}.
\end{equation}
Theorem~\ref{thm:main} will follow from Proposition~\ref{prop:multilinear-duality} if we can verify \eqref{eq:tensor:prod} and \eqref{eq:tensor:sum} with the chosen function $M$, the functions $S_{j}$ given by \eqref{eq:Sj}, and the function $G$ defined in \eqref{eq:G}.
In the remaining part of Section~\ref{sec:prf} we verify these conditions.

\subsection{Intersection multiplicity estimate}
\label{sec:intersection-multiplicity}
In order to verify \eqref{eq:tensor:sum} we use B\'ezout's theorem and the following change of variables.
\begin{lemma}[{\cite[Theorem 5.2]{arxiv:1510.09132}}]
\label{lem:intersection}
Let $Z_{1},\dotsc,Z_{m}$ be smooth submanifolds of $\R^{n}$ with $\sum_{j} (n-\dim Z_{j}) = n$.
Let also $U\subset (\R^{n})^{m-1}$ be measurable.
Then
\begin{multline}
\label{eq:intersection}
\int_{Z_{1}\times\dotsm\times Z_{m}} \one_{U}(x_{2}-x_{1},\dotsc,x_{m}-x_{1}) \abs[\Big]{\bigwedge_{j=1}^{m} N_{x_{j}}Z_{j}} \dif(x_{1},\dotsc,x_{m})
\leq\\
\int_{U} \abs{Z_{1}\cap (Z_{2}+v_{2}) \cap \dotsb \cap (Z_{m}+v_{m})} \dif (v_{2},\dotsc,v_{m}),
\end{multline}
where $\abs{\cdot}$ on the right-hand side denotes cardinality of the intersection (with multiplicity).
\end{lemma}

\begin{corollary}[{cf.\ \cite[Lemma 3.1]{MR2746348}}]
\label{cor:tube}
Let $p$ be a non-zero polynomial in $n$ variables.
Then for every cube $Q \subset \R^{n}$ with side length $1$ we have
\[
\abs{N(Z_{p}\cap Q)} \lesssim \deg p.
\]
\end{corollary}
\begin{proof}[Proof of Corollary~\ref{cor:tube} assuming Lemma~\ref{lem:intersection}]
It suffices to show that
\begin{equation}
\label{eq:dir-surface-in-cube}
\abs{\<v,N(Z_{p}\cap Q)\>} \lesssim \deg p
\end{equation}
for every unit vector $v\in\R^{n}$.
We apply Lemma~\ref{lem:intersection} with $m=2$, $Z_{1}=Z_{p}$, $Z_{2}$ being the line parallel to $v$ passing through the center of $Q$, and $U$ being the cube with side length $2$ centered at $0$.
Then the left-hand side of \eqref{eq:dir-surface-in-cube} is dominated by the left-hand side of \eqref{eq:intersection}.
On the other hand, the integrand on the right-hand side of \eqref{eq:intersection} is at most $\deg p$ for almost every $v_{2} \in \R^{n}$.
This implies \eqref{eq:dir-surface-in-cube}.
\end{proof}

\begin{proof}[Proof of Lemma~\ref{lem:intersection}]
We will identify $\abs{\wedge_{j} N_{x_{j}}Z_{j}}$ as the Jacobian of the change of coordinates
\[
(\R^{n})^{m} \supset Z_{1}\times\dotsm\times Z_{m} \to (\R^{n})^{m-1}, \quad
(x_{1},\dotsc,x_{n}) \mapsto (x_{2}-x_{1},\dotsc,x_{m}-x_{1}).
\]
To this end we note that
\[
\star \Bigl( \bigwedge_{l=1}^{n} \bigwedge_{j=2}^{m} (\dif x_{j,l}-\dif x_{1,l} ) \Bigr)
= \pm 1
\bigwedge_{l=1}^{n} \sum_{j=1}^{m} \dif x_{j,l},
\]
where $\star$ is the Hodge star operator, and where the sign depends on the choice of orientation in the definition of $\star$.
Indeed, both sides are simple wedge products, and each $1$-form on the right-hand side is orthogonal to each $1$-form on the left-hand side.
Also, one can verify that the norms (induced by the Euclidean norm) of the forms on both sides are equal to $m^{n/2}$.
From this follows equality up to the sign, see \cite[Appendix A]{MR3089762} for a formal proof.

Denoting by $T_{j}$ the normalized volume form on the tangent space $T_{x_{j}}Z_{j}$ and by $\pi_{j} : (\R^{n})^{m} \to \R^{n}$ the $j$-th coordinate projection, the Jacobian can be computed as
\begin{align*}
\abs[\Big]{ \< \bigwedge_{j=1}^{m} \pi_{j}^{*} T_{j}, \bigwedge_{l=1}^{n} \bigwedge_{j=2}^{m} (\dif x_{j,l}-\dif x_{1,l} ) \> }
&=
\abs[\Big]{ \< \star \bigwedge_{j=1}^{m} \pi_{j}^{*} T_{j}, \bigwedge_{l=1}^{n} \sum_{j=1}^{m} \dif x_{j,l} \> }\\
&=
\abs[\Big]{ \< \bigwedge_{j=1}^{m} \pi_{j}^{*} (\star T_{j}), \bigwedge_{l=1}^{n} \sum_{j=1}^{m} \dif x_{j,l} \> }\\
&=
\abs[\Big]{ \< \bigwedge_{j=1}^{m} (\star T_{j}), \bigwedge_{l=1}^{n} \dif x_{l} \> }\\
&=
\abs[\Big]{ \bigwedge_{j=1}^{m} (\star T_{j}) }.
\end{align*}
Indeed, to see the second identity we notice that
\[
\star \bigwedge_{j=1}^{m} \pi_{j}^{*} T_{j} = \bigwedge_{j=1}^{m} \pi_{j}^{*} (\star T_{j})
\]
holds for arbitrary $T_{j} \in \Lambda^{\dim Z_{j}}$, provided that we use on $(\R^{n})^{m}$ the product of the orientations on $\R^{n}$.

In the third identity we used the identity
\begin{equation}
\label{eq:3}
\< \bigwedge_{j=1}^{m} \pi_{j}^{*} (\tilde{T}_{j}), \bigwedge_{l=1}^{n} \sum_{j=1}^{m} \dif x_{j,l} \>
=
\< \bigwedge_{j=1}^{m} \tilde{T}_{j}, \bigwedge_{l=1}^{n} \dif x_{l} \>,
\end{equation}
which holds for arbitrary $\tilde{T}_{j} \in \Lambda^{n-\dim Z_{j}}$.
Indeed, by multilinearity it suffices to verify \eqref{eq:3} for simple wedge products $\tilde{T}_{j}=\wedge_{a=1}^{n-\dim Z_{j}} e_{l(j,a)}$.
In this case the left-hand side of \eqref{eq:3} equals
\[
\< \bigwedge_{j=1}^{m} \bigwedge_{a=1}^{n-\dim Z_{j}} \dif x_{j,l(j,a)}, \bigwedge_{l=1}^{n} \sum_{j=1}^{m} \dif x_{j,l} \>.
\]
Expanding the sums on the right-hand side, we see that this vanishes unless all $l(j,a)$ are distinct, in which case it equals the sign of the permutation $(l(j,a))$ that is ordered lexicographically, first by $j$ and then by $a$.
This description also applies to the right-hand side of \eqref{eq:3}.
\end{proof}

\begin{proof}[Proof of \eqref{eq:tensor:sum}]
By Lemma~\ref{lem:intersection} with $m=k_{j}+1$, $Z_{1} = H_{j}$, and $U=B(0,C)$, we have
\begin{align*}
\MoveEqLeft
\sum_{Q} \abs[\Big]{\<T_{Q}H_{j}, \bigwedge_{l=1}^{k_{j}} (N (Z_{l+1} \cap Q))\>}\\
&\leq
\int_{U} \abs{H_{j}\cap (Z_{2}+v_{2}) \cap \dotsb \cap (Z_{m}+v_{m})} \dif (v_{2},\dotsc,v_{m}).
\end{align*}
By B\'ezout's theorem \cite[Example 8.4.6]{MR1644323}, the cardinality on the right-hand side is either infinite or bounded by $(\deg H_{j}) \prod_{l=1}^{k_{j}} (\deg p_{l})$ if $Z_{l+1} = Z_{p_{l}}$ for some polynomials $p_{l}$.
But this cardinality is only infinite for a zero measure set of $v_{2},\dotsc,v_{m}$.
Hence
\[
\sum_{Q} \abs[\Big]{\<T_{Q}H_{j}, \bigwedge_{l=1}^{k_{j}} (N (Z_{p_{l}} \cap Q))\>}\\
\lesssim
(\deg H_{j}) \prod_{l=1}^{k_{j}} (\deg p_{l}).
\]
Integrating this estimate over $p_{1},\dotsc,p_{k_{j}}$ with respect to the measure $\sigma$ in each variable and using definitions \eqref{eq:Sj} and \eqref{eq:normal-measure}, we obtain \eqref{eq:tensor:sum}.
\end{proof}

\subsection{Wedge estimates}
\label{sec:wedge}
In order to apply \eqref{eq:Vis>M:scaled}, we need lower bounds on $\vol \B_{\s}$ in terms of related quantities.
This is roughly in the spirit of elementary lower bounds for the volume of polar bodies.
Here and later we will abbreviate $\B_{\mu} := \B_{\s_{\mu}}$.
\begin{lemma}[{cf.\ \cite[Theorem 3.1]{arxiv:1510.09132}}]
\label{lem:wedge}
Let $\mu$ be a measure on $\abs{\Lambda^{1}}$ with $\abs{\mu}<\infty$ and bounded $\B_{\mu}$.
Then
\[
1
\lesssim
\abs{\mu^{\wedge n}} \vol \B_{\mu}.
\]
\end{lemma}
\begin{proof}
Let $e_{1},\dotsc,e_{n}$ be an orthonormal basis of $\R^{n}$ consisting of vectors pointing in the directions of the axes of the John ellipsoid of $\B_{\mu}$.
Let $l_{1}\leq\dotsb\leq l_{n}$ be the lengths of these axes.
It suffices to show that
\[
\abs{\< \mu^{\wedge n} , l_{1}e_{1} \wedge\dotsm \wedge l_{n}e_{n} \>}
\gtrsim 1.
\]
Using a measurable selection $\abs{\Lambda^{1}} \to \R^{n}$, we can identify $\mu$ with a measure on $\R^{n}$.
Let $T : \R^{n} \to \R^{n}$ be the linear map with $Te_{j}=l_{j}e_{j}$.
Then $T$ is self-adjoint and
\begin{align*}
\abs{\< \mu^{\wedge n} , l_{1}e_{1} \wedge\dotsm \wedge l_{n}e_{n} \>}
&=
\int \abs{\< v_{1} \wedge\dotsm\wedge v_{n}, T e_{1} \wedge\dotsm\wedge Te_{n} \>} \dif\mu(v_{1}) \dotsm \dif\mu(v_{n})\\
&=
\int \abs{\det \< v_{i} , T e_{j} \>_{ij} } \dif\mu(v_{1}) \dotsm \dif\mu(v_{n})\\
&=
\int \abs{\det \< Tv_{i} , e_{j} \>_{ij} } \dif\mu(v_{1}) \dotsm \dif\mu(v_{n})\\
&=
\int \abs{\det \< v_{i} , e_{j} \>_{ij} } \dif\tilde\mu(v_{1}) \dotsm \dif\tilde\mu(v_{n}),
\end{align*}
where $\tilde\mu$ is the pushforward of $\mu$ under $T$.
Hence it suffices to show that
\begin{equation}
\label{eq:1}
\int \abs{v_{1} \wedge \dotsm \wedge v_{n}} \dif\tilde\mu(v_{1}) \dotsm \dif\tilde\mu(v_{n})
\gtrsim 1.
\end{equation}
Note that
\[
\s_{\tilde\mu}(u)
=
\int \abs{\<u,v\>} \dif\tilde\mu(v)
=
\int \abs{\<u,Tv\>} \dif\mu(v)
=
\s_{\mu}(Tu).
\]
It follows that the John ellipsoid of $\B_{\tilde\mu}$ is the unit ball, hence $\s_{\tilde\mu}(u) \sim \norm{u}$.
This in turn implies that $\abs{\tilde\mu}\sim 1$, and that the measure $\abs{v}\dif\tilde\mu(v)$ cannot concentrate in a small angular neighborhood of the hyperplane $u^{\perp}$ for any $u \in \R^{n}\setminus\Set{0}$.
Hence we can select $n$ transverse cones in $\R^{n}$ whose projections to $\abs{\Lambda^{1}}$ support a fixed proportion of $\abs{v}\dif\tilde\mu(v)$.
Restricting each variable on the left-hand side of \eqref{eq:1} to one of these cones we see that \eqref{eq:1} holds.
\end{proof}

\begin{lemma}[{\cite[Theorem 1]{MR0101508}}]
\label{lem:convex-volume-split}
Let $K\subset\R^{n}$ be a convex body and $T\subset\R^{n}$ a $k$-dimensional affine subspace.
Then
\begin{equation}
\label{eq:convex-volume-split}
\vol_{n-k} (K + T) \vol_{k}(K \cap T)
\leq
\binom{n}{k} \vol_{n} K,
\end{equation}
where $K+T := \Set{x + T \given x\in K} \subset \R^{n}/T$, and the latter space is equipped with the Euclidean volume with respect to the quotient metric.
\end{lemma}
For centrally symmetric convex bodies a converse inequality to \eqref{eq:convex-volume-split} also holds, but we will not use this fact.

\begin{corollary}
\label{cor:wedge:subspace}
Let $T \subset \R^{n}$ be a $k$-dimensional affine subspace and $\mu$ a measure on $\abs{\Lambda^{1}}$.
Then
\[
\vol_{n-k}(\B_{\mu} + T)
\lesssim
\abs{\<\mu^{\wedge k}, T\>} \vol_{n} \B_{\mu}.
\]
\end{corollary}
\begin{proof}
The conclusion only depends on the equivalence class of $T$ modulo translation, so we may assume $0 \in T$.
By Lemma~\ref{lem:convex-volume-split} with $K=\B_{\mu}$, it suffices to show
\begin{equation}
\label{eq:2}
1 \lesssim \vol_{k}(\B_{\mu} \cap T) \abs{\<\mu^{\wedge k},T\>}.
\end{equation}
Let $\tilde\mu$ be the pushforward of $\mu$ under the orthogonal projection onto $T$, which is well-defined modulo $\pm 1$.
Then $\B_{\mu} \cap T = \B_{\tilde\mu}$ and $\abs{\<\mu^{\wedge k},T\>} = \abs{\tilde\mu^{\wedge k}}$.
The claim \eqref{eq:2} now follows from Lemma~\ref{lem:wedge} with $n$ replaced by $k$ and $\mu$ replaced by $\tilde\mu$.
\end{proof}

\begin{lemma}[{cf.\ \cite[Corollary 7.6]{arxiv:1510.09132}}]
\label{lem:BL-est}
Let $T_{1},\dotsc,T_{m} \subset\R^{n}$ be proper affine subspaces of dimensions $k_{1},\dotsc,k_{m}$.
Let $\mu$ be a measure on $\abs{\Lambda^{1}}$ such that $\frac{c}{R}\B \subseteq \B_{\mu} \subseteq C\B$.
Then
\begin{equation}
\label{eq:BL-est}
\vol(\B_{\mu})^{1-\sum_{j=1}^{m} p_{j}}
\leq C^{P}
\BL(\vec T, \vec{p}, (1/R,1))
\prod_{j=1}^{m} \abs{\<T_{j}, \mu^{\wedge k_{j}}\>}^{p_{j}}.
\end{equation}
\end{lemma}
\begin{proof}
By Definition~\ref{def:BL} and Corollary~\ref{cor:wedge:subspace}, we have
\begin{align*}
\vol(\B_{\mu})
&=
\int_{\R^{n}} \one_{\B_{\mu}} \dif x\\
&\leq
\int_{\R^{n}} \prod_{j=1}^{m} (\one_{\B_{\mu} + T_{j}}(x+T_{j}))^{p_{j}} \dif x\\
&\lesssim
\BL(\vec T, \vec{p}, (1/R,1)) \prod_{j=1}^{m} \bigl( \int_{\R^{n}/T_{j}} \one_{\B_{\mu} + T_{j}} \bigr)^{p_{j}}\\
&\leq
\BL(\vec T, \vec{p}, (1/R,1)) \prod_{j=1}^{m} \Bigl( C \abs[\big]{\<T_{j}, \mu^{\wedge k_{j}}\>} \vol_{n}(\B_{\mu}) \Bigr)^{p_{j}}.
\qedhere
\end{align*}
\end{proof}

\begin{proof}[Proof of \eqref{eq:tensor:prod}]
By definition \eqref{eq:Fotimes} of the Fremlin tensor product norm, we obtain
\begin{align*}
\prod_{j=1}^{m} (R^{k_{j}} S_{j}(Q))^{p_{j}}
&=
\prod_{j=1}^{m} \bigl( \int_{Q\cap H_{j}} \abs{\<T_{x_{j}}H_{j}, \mu_{Q}^{\wedge k_{j}}\>} \dif x_{j} \bigr)^{p_{j}}
& \text{by \eqref{eq:Sj}}\\
&=
\prod_{j=1}^{m} \norm{ \abs{\<T_{x_{j}}H_{j}, \mu_{Q}^{\wedge k_{j}}\>}^{p_{j}/P} }_{L^{P/p_{j}}_{x_{j}}(Q\cap H_{j})}^{P}\\
&\geq
\norm{ \prod_{j=1}^{m} \abs{\<T_{x_{j}}H_{j}, \mu_{Q}^{\wedge k_{j}}\>}^{p_{j}/P} }_{\Fotimes_{j=1}^{m} L^{P/p_{j}}_{x_{j}}(H_{j} \cap Q)}^{P}
& \text{by \eqref{eq:Fotimes}.}
\end{align*}
The measure $\mu_{Q} = \mu_{\sigma,Q}$ satisfies $\B_{\mu} \subseteq C\B$ by \eqref{eq:Bmu-subset-B}.
On the other hand, since $\sigma$ is supported on a set of polynomials of degree $\lesssim R$, it follows from Corollary~\ref{cor:tube} that $\s_{\sigma,Q}(v) \lesssim R\abs{v}$, so that $\frac{c}{R} \B \subseteq \B_{\mu_{Q}}$ for some small constant $c>0$ depending only on $n$.
Therefore, we can continue the above chain of inequalities with
\begin{align*}
&\geq C^{-P}
\vol(\B_{\mu_{Q}})^{1-P} \norm{
\BL( \widevec{T_{x_{j}}H_{j}}, \vec{p}, (1/R,1) )^{-1/P}
}_{\Fotimes_{j=1}^{m} L^{P/p_{1}}_{x_{j}}(H_{j} \cap Q)}^{P}
& \text{by Lemma~\ref{lem:BL-est}}\\
&= C^{-P}
\vol(\B_{\mu_{Q}})^{1-P} R^{n-\sum_{j=1}^{m} n_{j}p_{j}}
G(Q)
&\text{by \eqref{eq:BL-scaling} and \eqref{eq:G}}.
\end{align*}
Using that $k_{j}+n_{j}=n$, we can rearrange this inequality as
\[
C^{P} \prod_{j=1}^{m} S_{j}(Q)^{p_{j}}
\geq
(R^{n} \vol(\B_{\mu_{Q}}))^{1-P} G(Q).
\]
By \eqref{eq:Vis>M:scaled}, we have
\[
R^{n} \vol(\B_{\mu_{Q}})
=
R^{n} \vol(\B_{\s_{p,Q}})
\leq
C M(Q)^{-1},
\]
assuming that $M(Q)>0$.
Since $P \geq 1$ by the hypothesis \eqref{eq:P}, the last two displays imply \eqref{eq:tensor:prod}.
\end{proof}

\printbibliography
\end{document}